\documentclass[reqno,a4paper,11pt]{article}
\usepackage{amsmath,amsthm,dsfont,amsfonts,amssymb,fancyhdr}
\usepackage{enumerate}
\usepackage[usenames,dvipsnames]{color}
\usepackage{bbm,soul,supertabular,longtable,verbatim,extarrows}
\usepackage{titlesec}
\usepackage{graphicx}
\usepackage{BOONDOX-cal}
\usepackage{cite}
\usepackage{tikz}
\usepackage{booktabs,multirow,makecell}
\usepackage{authblk}
\usetikzlibrary[trees]

\usepackage{lipsum}
\usepackage{mathrsfs}
\usepackage{indentfirst}
\usepackage[colorlinks=true, allcolors=blue]{hyperref}
\usepackage[top=2.4cm,bottom=2.2cm,left=2.6cm,right=2cm]{geometry}

\newtheorem{theorem}{Theorem}[section]
\newtheorem{proposition}[theorem]{Proposition}
\newtheorem{lemma}[theorem]{Lemma}

\newtheorem{conjecture}[theorem]{Conjecture}
\newtheorem{claim}[theorem]{Claim}
\theoremstyle{definition}
\newtheorem{definition}[theorem]{Definition}
\newtheorem{remark}[theorem]{Remark}

\soulregister\cite7
\soulregister\citep7
\soulregister\citet7
\soulregister\ref7
\soulregister\pageref7

\begin{document}
\title{\textbf{Two characterizations of the grid graphs}}
\author[a]{Brhane Gebremichel}
\author[b]{Meng-Yue Cao\footnote{Corresponding author.}}
\author[a,c]{Jack H. Koolen}
\affil[a]{\footnotesize{School of Mathematical Sciences, University of Science and Technology of China, 96 Jinzhai Road, Hefei, 230026, Anhui, PR China.}}
\affil[b]{\footnotesize{School of Mathematical Sciences, Beijing Normal University, 19 Xinjiekouwai Street, Beijing, 100875, PR China.}}
\affil[c]{\footnotesize{CAS Wu Wen-Tsun Key Laboratory of Mathematics, University of Science and Technology of China, 96 Jinzhai Road, Hefei, Anhui, 230026, PR China}}
\date{}
\maketitle
\newcommand\blfootnote[1]{%
\begingroup
\renewcommand\thefootnote{}\footnote{#1}%
\addtocounter{footnote}{-1}%
\endgroup}
\blfootnote{E-mail addresses: {\tt brhaneg220@mail.ustc.edu.cn} (B. Gebremichel), {\tt cmy1325@163.com} (M.-Y. Cao), {\tt koolen@ustc.edu.cn} (J.H. Koolen).}

\begin{abstract}
In this paper we give two characterizations of the $p \times q$-grid graphs as co-edge-regular graphs with four distinct eigenvalues.
\end{abstract}


\section{Introduction}
All graphs mentioned in this paper are finite, undirected and simple. For undefined notations, see \cite{BCN} and \cite{BH11}.
The eigenvalues of a graph are the eigenvalues of its adjacency matrix in this paper.
Recall that a co-edge-regular graph with parameters $(n, k, c)$ is a $k$-regular graph with $n$ vertices, such that any two distinct non-adjacent vertices have exactly $c$ common neighbours.

In this paper, we continue our study on co-edge-regular graphs with four distinct eigenvalues. A motivation to study these graphs comes from the lecture note \cite{Terwilliger}. In these note, Terwilliger shows that any local graph of a thin $Q$-polynomial distance-regular graph is co-edge-regular and has at most five distinct eigenvalues. So it is interesting to study co-edge-regular graphs with a few distinct eigenvalues. Another motivation is as follows: strongly regular graphs have attracted a lot of attention, see for example \cite{VanLint1984}. On the other hand, there are only a few papers on regular graphs with four eigenvalues, see for example \cite{vanDam1995,vanDam1999,vanDam1998,Huang2017}. We think that connected co-edge-regular graphs are an interesting class of graphs.

Tan, Koolen and Xia \cite{Tan.2020} gave the following conjecture.

\begin{conjecture}\label{conjecture by TKX}
Let $G$ be a connected co-edge-regular graph with parameters $(n,k,c)$ having four distinct eigenvalues. Let $m \geqslant 2$ be an integer. Then there exists a constant $n_m$ such that, if $\theta_{\min}(G) \geqslant -m$ and $n \geqslant n_m$ and $k<n-2-\frac{(m-1)^2}{4}$, then either $G$ is the $s$-clique extension of a strongly regular graph for $2 \leqslant s \leqslant m-1$ or $G$ is a $p \times q$-grid with $ p > q \geqslant 2$.
\end{conjecture}
The first result on co-edge-regular graphs was shown by Brouwer, Cohen and Neumaier \cite[Lemma 1.1.3]{BCN}. They
showed that a co-edge-regular graph with parameters $(n,k,1)$ is strongly regular. They gave some more conditions such that a co-edge-regular graph satisfying these conditions is strongly regular. Hayat, Koolen and Riaz \cite{Hayat2019} showed that any co-edge-regular graph cospectral with the $s$-clique extension of the $t\times t$-grid is the $s$-clique extension of the $t\times t$-grid, if $t$ is much larger than $s$. Tan et al. \cite{Tan.2020} showed that any co-edge-regular graph cospectral with the $s$-clique extension of the triangular graph $T(t)$ is the $s$-clique extension of $T(t)$. In this paper, we will concentrate on co-edge-regular graphs with parameters $(n,k,2)$ having exactly four distinct eigenvalues. A $p \times q$-grid is the line graphs of the complete bipartite graph $K_{p,q}$. In other words it is the cartesian product of the complete graphs $K_p$ and $K_q$. For connected co-edge-regular graphs with four distinct eigenvalues, we obtain the following two characterizations of the $p\times p$-grid.

\begin{theorem}\label{fourev}
Let $G$ be a co-edge-regular graph with parameters $(n, k, 2)$ with distinct eigenvalues $k=\theta_0  > \theta_1 > \theta_2 > \theta_3$.
Let $\ell := 2(\sum\limits_{i=1}^3 \theta_i ) +\frac{\prod_{i=1}^3(k-\theta_i)}{n} -2(k-2)$.
\begin{enumerate}
\item If $\ell\geqslant \frac{3}{4}k$, then $G$ is a $p\times q$-grid, where $p >q \geqslant 2$, $p+q=k+2$ and $\ell = k-2$ or the $2$-clique extension of $C_5$.
\item If $\theta_3 \geqslant -3$ and $k \geqslant 120$, then $G$ is a $p\times q$-grid, where  $p >q \geqslant 2$, $p+q=k+2$ and $\ell = k-2$.
\end{enumerate}
\end{theorem}

Now we introduce a class of co-edge-regular graphs that gives a combinatorial generalization of co-edge-regular graphs with exactly four distinct eigenvalues. Let $G$ be a graph.
Let $a_{xy}$ denote the number of common neighbours of two adjacent vertices $x$ and $y$ in $G$. A \emph{strongly co-edge-regular graph} $G$ with parameters $(n,k,c,\ell)$ is a co-edge-regular graph with parameters $(n,k,c)$ such that for any two distinct non-adjacent vertices $x$ and $z$, $\sum_{y\sim x,y\sim z} a_{xy}=\ell$ holds.
Note that there are many strongly co-edge-regular graphs, for example, the complement of a distance-regular graph of diameter at least $2$ is strongly co-edge-regular.

A co-edge-regular graph with exactly four distinct eigenvalues is strongly co-edge-regular and walk-regular, as we will show in Section \ref{4.2}.

We will use the concept of strongly co-edge-regular graphs to show the following two results.

\begin{theorem}\label{grid}
Let $G$ be a walk-regular and strongly co-edge-regular graph with parameters $(n,k$, $2,\ell)$. If $\ell\geqslant \frac{3}{4}k$, then $G$ is a $p\times q$-grid, where $p+q=k+2$ and $\ell = k-2$ or the $2$-clique extension of $C_5$.
\end{theorem}

When we moreover assume that the smallest eigenvalue is at least $-3$, we can remove the bound on $\ell$.

\begin{theorem}\label{main1}
Let $G$ be a walk-regular and strongly co-edge-regular graph with parameters $(n,k$, $2,\ell)$ with smallest eigenvalue $\theta_{\min}$ at least $-3$. If $k\geqslant 120$, then $G$ is a $p\times q$-grid, where $p+q=k+2$ and $\ell = k-2$.
\end{theorem}

\begin{remark}
\begin{enumerate}
\item The $2$-clique extension of the pentagon $C_5$ is a co-edge-regular graph with parameters $(10,5,2)$ with exactly four distinct eigenvalues and smallest eigenvalue $-\sqrt{5}$.
\item The $2$-clique extension of the Petersen graph is co-edge-regular with parameters $(20,7,2)$ with exactly four distinct eigenvalues and smallest eigenvalue $-3$.
\end{enumerate}
\end{remark}

Theorem \ref{fourev} follows directly from Theorem \ref{grid} and \ref{main1}, as a co-edge-regular graph with four distinct eigenvalues is walk regular and strongly co-edge-regular graph.

This paper is organized as follows. In Section \ref{Pre} we give preliminaries. In Section \ref{section3} we give some results on co-edge-regular and strongly co-edge-regular graphs. We show Theorem \ref{main1} and Theorem \ref{fourev} in Section \ref{MainResults}.

\section{Preliminaries}\label{Pre}

\subsection{Graphs}

A graph $G$ is an ordered pair $(V(G),E(G))$, where $V(G)$ is a finite set and $\displaystyle E(G)\subseteq \binom{V(G)}{2}$. The set $V(G)$ (resp.\ $E(G)$) is called the \emph{vertex set} (resp. \emph{edge set}) of $G$. A \emph{subgraph} of a graph $G$ is a graph $H$ such that $V(H)\subseteq V(G)$ and $E(H)\subseteq E(X)$. If $E(H)=\binom{V(H)}{2}\cap E(G)$, then we say $H$ is an \emph{induced subgraph} of $G$. The disjoint union of the graphs $G_1$ and $G_2$ is denoted by $G_1\dot\cup G_2$. If $\{x,y\}$ is an edge in $E$, then we say the vertices $x,y$ are adjacent, denoted by $x\sim y$, and otherwise, we say that $x,y$ are not adjacent, denoted by $x\not\sim y$ . The \emph{complement} $\overline{G}$ of a graph $G$ has the same vertex set as $G$, where distinct vertices $x$ and $y$ are adjacent in $\overline{G}$ if and only if they are not adjacent in $G$. The \emph{adjacency matrix} of $G$, denoted by $A(G)$, is a symmetric $(0,1)$-matrix indexed by $V(G)$, such that $(A(G))_{xy}=1$ if and only if $x\sim y$. The \emph{eigenvalues} of $G$ are the eigenvalues of $A(G)$. The
\emph{spectrum} of a graph $G$ is the list of the eigenvalues of $A(G)$ together with their multiplicities.

Let $G$ be a graph. For a vertex $x\in V(G)$, we denote by $N_G(x)$ the set of the neighbours of $x$ in $G$, and call the subgraph induced on $N_G(x)$ the \emph{local graph} of $x$ in $G$. We denote by $N_G(x,y)$ the set of common neighbours of $x,y$ in $G$. We write $a_{xy}$ for the cardinality of $N_G(x,y)$, if $x,y$ are adjacent. A graph $G$ is called bipartite if its vertex set can be partitioned into two parts $V_1$ and $V_2$ such that every edge has one end in $V_1$ and one in $V_2$. If every vertex in $V_1$ is adjacent to all vertices in $V_2$, then we say $G$ is a complete bipartite graph, and also denoted by $K_{|V_1|,|V_2|}$. A graph is complete, if every pair of distinct vertices is adjacent. We call a subgraph is a \emph{clique}, if it is isomorphic to a complete graph. We say a clique with $s$ vertices an $s$-clique. The cardinality of a maximum clique in a graph $G$ is called the \emph{clique number} of $G$, and is denoted by $\omega(G)$. A graph $G$ is called \emph{$k$-regular} if every vertex in $G$ has $k$ neighbours.
\begin{definition}
Let $G$ be a $k$-regular graph on $n$ vertices that is neither complete nor empty. Then $G$ is said to be
\begin{enumerate}
  \item \emph{co-edge-regular} with parameters $(n,k,c)$, if any pair of distinct non-adjacent vertices have $c$ common neighbours.
  \item \emph{strongly regular} with parameters $(n, k, a, c)$, if any two adjacent vertices have $a$ common neighbours and any pair of distinct non-adjacent vertices have $c$ common neighbours.
   \item \emph{walk-regular}, if for all nonnegative integers $r$, all the diagonal entries of $A^r$ are the same, where $A$ is the adjacency matrix of $G$.
\end{enumerate}
\end{definition}

For a positive integer $s$, the \emph{$s$-clique extension} of a graph $G$ is the graph $\tilde{G}$ obtained from $G$ by replacing each vertex $x \in V(G)$ by a clique $\tilde{X}$ with s
vertices, such that $\tilde{x} \sim \tilde{y}$ (for $\tilde{x} \in \tilde{X}, \tilde{y} \in \tilde{Y}$) in $\tilde{G}$ if and only if $ x \sim y$ in $G$. If $\tilde{G}$ is the $s$-clique extension of $G$,
then $\tilde{G}$ has adjacency matrix $\mathbf{J}_s \otimes (A(G) + \mathbf{I}_n) - \mathbf{I}_{sn}$, where $\mathbf{I}$ is the identity matrix and $\mathbf{J}$ is the all-ones matrix. If $G$
has spectrum $\{\theta_0^{m_0}, \theta_1^{m_1}, \ldots , \theta_t^{m_t} \}$, then the spectrum of $\tilde{G}$ is
\begin{equation*}
  \{ (s(\theta_0 +1)-1)^{m_0}, (s(\theta_1 +1)-1)^{m_1}, \ldots , (s(\theta_t+1)-1)^{m_t}, (-1)^{(s-1)(m_0 +m_1 + \cdots + m_t)} \}.
\end{equation*}

\subsection{Interlacing}

If $M$ (resp. $N$) is a real symmetric $m\times m$ (resp. $n\times n$) matrix, let $\eta_1(M)\geqslant\eta_2(M)\geqslant\cdots\geqslant\eta_m(M)$ (resp. $\eta_1(N)\geqslant\eta_2(N)\geqslant\cdots\geqslant\eta_n(N)$) denote the eigenvalues of $M$ (resp. $N$) in nonincreasing order. Assume $m\leqslant n$. We say that the eigenvalues of $M$ \emph{interlace} the eigenvalues of $N$, if $\eta_{n-m+i}(N)\leqslant\eta_i(M)\leqslant\eta_i(N)$ for each $i=1,\ldots,m$. The following result is a special case of interlacing.


\begin{lemma}[{cf.\cite[Theorem 9.1.1]{GD01}}]\label{interlacing}
Let $B$ be a real symmetric $n\times n$ matrix and $C$ a principal submatrix of $B$ of order $m$, where $m<n$. Then the eigenvalues of $C$ interlace the eigenvalues of $B$.
\end{lemma}

As an easy consequence of Lemma \ref{interlacing}, we have the following proposition.
\begin{proposition}\label{subgraph}
Let $G$ be a graph and $H$ a proper induced subgraph of $G$. Denote by $\theta_{\min}(G)$ (resp. $\theta_{\min}(H)$) the smallest eigenvalue of $G$ (resp. $H$). Then $\theta_{\min}(G)\leqslant \theta_{\min}(H)$.
\end{proposition}

Let $G=(V,E)$ be a graph and $\pi:=\{V_1,\ldots,V_r\}$ be a partition of $V$. We say $\pi$ is an \emph{equitable partition} with respect to $G$ if the number of neighbours in $V_j$ of a vertex $u$ in $V_i$ is a constant $q_{ij}$, independent of $u$. For an equitable partition $\pi$ with respect to $G$, the quotient matrix $Q$ of $G$ with respect to $\pi$ is defined as $Q=(q_{ij})_{1\leqslant i,j\leqslant r}$.

\begin{lemma}[{cf.\cite[Theorem 9.3.3]{GD01}}]\label{quotientmatrix}
Let $G$ be a graph. If $\pi$ is an equitable partition of $G$ and $Q$ is the quotient matrix with respect to $\pi$ of $G$, then every eigenvalue of $Q$ is an eigenvalue of $G$.
\end{lemma}

For a graph $G$, let $C(G)$ be the \emph{cone} of $G$, that is, $C(G)$ is obtained by adding a vertex to $G$ and joining it to all vertices of $G$.

\begin{lemma}\label{forbiddensubgraphs}
Let $G$ be a graph with smallest eigenvalue at least $-3$. Then none of the following graphs is an induced subgraph of $G$.
\begin{enumerate}
  \item Connected bipartite graphs with order at least $11$ and containing an induced $K_{1,9}$;
  \item Graphs $C(2K_s\dot\cup tK_1)$, where $(s+2)(t-3)>12$;
  \item Graphs $C(2K_{15}\dot\cup K_3\dot\cup 2K_1)$, $C(2K_{21}\dot\cup K_{11}\dot\cup K_1)$, 
  $C(C(2K_{13})\dot\cup K_{13})$, $C(C(3K_5))$.
\end{enumerate}
\end{lemma}
\begin{proof}
Let $G$ be a graph with smallest eigenvalue at least $-3$.

 $(\mathrm{i})$ Let $B$ be a connected bipartite graph with order $n\geqslant11$. Assume that $B$ contains an induced $K_{1,9}$. Denote by $\theta_{\max}(B)$ the largest eigenvalue and $\theta_{\min}(B)$ the smallest eigenvalue of $B$. By the Perron-Frobenius Theorem \cite[Theorem 3.1.1]{BCN}, we have $\theta_{\max}(B)>3$, as the largest eigenvalue of $K_{1,9}$ is $3$. Since $B$ is a bipartite graph, we obtain $\theta_{\min}(B)=-\theta_{\max}(B)<-3$. It follows by Lemma \ref{interlacing} that $G$ does not contain $B$ as an induced subgraph.

 $(\mathrm{ii})$ Assume that $G$ contains $C(2K_s\dot\cup tK_1)$, say $H$, as an induced subgraph for some integers $s,t$. By Lemma \ref{interlacing}, we have the smallest eigenvalue of $H$ is at least $-3$. Let $u$ be the vertex of valency $2s+t$ in $H$. Let $V_1$ be the set of vertices of valency $s$ in $H$ and $V_2=V(H)-\{u\}-V_1$. Consider a partition $\pi=\{\{u\}, V_1, V_2\}$ of $H$. The partition $\pi$ is equitable with quotient matrix $Q$:
\begin{gather*}
Q=\begin{pmatrix}
0 & 2s & t \\ 1 & s-1 & 0 \\ 1 & 0 & 0
\end{pmatrix}.
\end{gather*}
Note that $\det(Q+3\mathbf{I})=-(s+2)(t-3)+12$. By Lemma \ref{quotientmatrix}, we see that the smallest eigenvalue of $Q$ is at least $-3$. Hence, we have $(s+2)(t-3)\leqslant 12$, as $\det(Q+3\mathbf{I})\geqslant0$. This shows $G$ does not contain $C(2K_s\dot\cup tK_1)$ as an induced subgraph for $(s+2)(t-3)>12$. 

$(\mathrm{iii})$ By using a similar method as in the proof for $(\mathrm{ii})$, we obtain $(\mathrm{iii})$.
\end{proof}

\subsection{Strongly regular graphs}
A strongly regular graph $G$ with at least $2$ vertices is called \emph{primitive} if both $G$ and its complement are connected. Note that, if $G$ is primitive strongly regular with parameters $(n,k,a,c)$, then $0<c<k$.
A \emph{conference graph} is a strongly regular graph with parameters $(4c+1,2c,c-1,c)$, where $c$ is a positive integer.

\begin{lemma}[{cf.\cite[Lemma 1.2]{Neumaier.1979}}]\label{conferencegraph}
Let $G$ be a strongly regular graph with parameters $(n,k,a,c)$ and eigenvalues $k>\theta_1>\theta_2$. Then $G$ is a conference graph or both $\theta_1,\theta_2$ are integers.
\end{lemma}

\begin{lemma}[{cf.\cite[Sections 10.2 and 10.3]{GD01}}]\label{SRGspectrum}
Let $G$ be a strongly regular graph with parameters $(n,k,a,c)$, where $k>c$. Then $G$ has exactly three distinct eigenvalues $k>\theta>\tau$ satisfying
\begin{align*}
\theta & =\frac{(a-c)+\sqrt{(a-c)^2+4(k-c)}}{2},\\
\tau & =\frac{(a-c)-\sqrt{(a-c)^2+4(k-c)}}{2}.
\end{align*}
Moreover, $m_\tau-m_\theta=\frac{2k+(n-1)(a-c)}{\sqrt{(a-c)^2+4(k-c)}}$, where $m_\theta$ and $m_\tau$ are the respective multiplicities of $\theta,\tau$.
\end{lemma}

\begin{lemma}[{cf.\cite[Theorem 4.7]{Neumaier.1979}}]\label{secondev}
Let $G$ be a strongly regular graph with parameters $(n,k,a,c)$ and eigenvalues $k>\theta_1>\theta_2$, where $\theta_2<-1$ is an integer. If $c\notin\{\theta_2(\theta_2+1),\theta_2^2\}$, then
\begin{equation*}
  \theta_1\leqslant\frac{\theta_2(\theta_2+1)(c+1)}{2}-1.
\end{equation*}
\end{lemma}

\begin{lemma}[{cf.\cite[Corollary 3.12.3 and Theorem 3.12.4]{BCN}}]\label{atleast-2}
Let $G$ be a connected regular graph with smallest eigenvalue $\theta_{\min}$.
\begin{enumerate}
\item If $\theta_{\min}>-2$, then $G$ is a clique or an odd cycle.
\item If $G$ is a strongly regular graph and $\theta_{\min}=-2$, then $G$ is a triangle graph $T(n)$ ($n\geqslant5$), a square grid $n\times n$ ($n\geqslant3$), a complete multipartite graph $K_{n\times2}$ ($n\geqslant2$), or one of the graphs of Petersen, Clebsch, Schl\"{a}fli, Shrikhande, or Chang.
\end{enumerate}
\end{lemma}

Table \ref{table} below states the parameters of $K_n, C_5$ and all the graphs in Lemma \ref{atleast-2} (ii).

\begin{table}[h]
  \centering
  \begin{tabular}{c|c}
  \hline\hline
  Graph & Parameters \\ \hline
  $K_n$ & $(n,n-1,n-2,0)$ \\ \hline
  the pentagon $C_5$ & $(5,2,0,1)$ \\\hline
  $T(n)$ ($n\geqslant5$) & $(\frac{n(n-1)}{2},2n-4,n-2,4)$ \\\hline
  $n\times n$-grid ($n\geqslant3$) & $(n^2,2n-2,n-2,2)$ \\\hline
  $K_{n\times2}$ ($n\geqslant2$) & $(2n,2(n-1),2(n-2),2(n-1))$ \\\hline
  the Petersen graph & $(10,3,0,1)$ \\\hline
  the Clebsch graph & $(16,10,6,6)$ \\\hline
  the Schl\"{a}fli graph & $(27,16,10,8)$ \\\hline
  the Shrikhande graph & $(16,6,2,2)$ \\\hline
  the Chang graphs & $(28,12,6,4)$ \\\hline
  \hline
\end{tabular}
  \caption{Parameters of $K_n, C_5$ and strongly regular graphs with smallest eigenvalue $-2$}\label{table}
\end{table}

\subsection{Terwilliger graphs}
A \emph{Terwilliger graph} is a non-complete graph $G$ such that, for any two vertices $x,y$ at distance $2$, the subgraph induced by $N_G(x,y)$ forms a clique of size $c$ (for some fixed $c\geqslant 0$).
\begin{lemma}[{cf.\cite[Proposition 1.16.2]{BCN}}]\label{s-clique}
  Let $G$ be a connected co-edge-regular Terwilliger graph. Then $G$ is the $s$-clique extension of a strongly regular graph, where $s$ is a positive integer.
\end{lemma}

\begin{proposition}[{cf.\cite[Proposition 6 (1) and (2)]{Koolen.2012}}] \label{k bound}
\begin{enumerate}
  \item  Let $G$ be a connected non-complete strongly regular with parameters $(n,k,a,1)$. If $k < 7(a+1)$, then $G$ is either the pentagon or the Petersen graph.
  \item Let $G$ be a connected non-complete strongly regular Terwilliger graph with $n$ vertices and valency $k$. If $k \leqslant 7a$, then $G$ is either the pentagon or the Petersen graph.
\end{enumerate}
\end{proposition}

\begin{lemma}\label{noquadrangles}
Let $G$ be a strongly co-edge-regular graph with parameters $(n,k,2, \ell)$. If $\ell \geqslant \frac{2}{7}k$, then $G$ contains an induced quadrangle or $G$ is the $2$-clique extension of the pentagon or the Petersen graph.
\end{lemma}
\begin{proof}
Let $G$ be a strongly co-edge-regular graph with parameters $(n,k,2, \ell)$ which does not contain induced quadrangles. Then $G$ is a Terwilliger graph of diameter $2$ and hence $G$ is the $s$-clique extension of a non-complete strongly regular graph $H$ for some integer $s$, by Lemma \ref{s-clique}.

If $s=1$, then  $G$ is a strongly regular Terwilliger graph with parameters $(n,k,a,2)$ and $\ell =2a$. As $\ell \geqslant \frac{2}{7}k$ we have $k \leqslant 7a$. Thus, by Proposition \ref{k bound} $(ii)$ , we see that no such graph exists.

If $s=2$, then $G$ is the $2$-clique extension of a strongly regular graph $H$ with parameters $(n_H,k_H, a_H,1)$. It is not hard to see that
\begin{equation*}
  n_H= \frac{n}{2},  \text{  }  k_H = \frac{k-1}{2}, \text{ and } a_H = \frac{\ell-4}{4}.
\end{equation*}
As $\ell \geqslant \frac{2}{7}k$, we have $k_H < 7(a_H+1)$ and hence $H$ is the pentagon or the Petersen graph, by Proposition \ref{k bound}$(i)$. This completes the proof of the lemma.

\end{proof}

\section{Co-edge-regular graphs and strongly co-edge-regular graphs}\label{section3}
In this section, we state some results for co-edge-regular graphs and strongly co-edge-regular graphs.

\begin{lemma}\label{twoconstants}
Let $G$ be a walk-regular and co-edge-regular graph with parameters $(n,k,c)$. Let $x$ be a vertex of $G$ and $a_{xy}$ the number of common neighbours of $x,y$ for $y\in N_G(x)$. Then $\sum\limits_{y\sim x}a_{xy}$ and $\sum\limits_{y\sim x}a_{xy}^2$ only depend on the spectrum of $G$.
\end{lemma}
\begin{proof}
Let $G$ be a walk-regular graph and co-edge-regular graph with parameters $(n,k,c)$. Let $A$ be the adjacency matrix of $G$. As $G$ is walk-regular, for any vertex $x$, the numbers $(A^3)_{xx}$ and $(A^4)_{xx}$ only depend on the spectrum of $G$.

As $\sum\limits_{y\sim x}a_{xy}=(A^3)_{xx}$, we see that $\sum\limits_{y\sim x}a_{xy}$ only depends on the spectrum of $G$ for any vertex $x$ of $G$.

Note that
\begin{equation*}
  (A^4)_{xx}  =k^2+\sum\limits_{y\sim x}a^2_{xy}+(n-k-1)c^2,
\end{equation*}
as $G$ is co-edge-regular with parameters $(n,k,c)$. Hence,  $\sum\limits_{y\sim x}a_{xy}^2$ only depends on the spectrum of $G$ for any vertex $x$ of $G$, as
$\sum\limits_{y\sim x}a_{xy}$ only depends on the spectrum of $G$.

\end{proof}

\begin{lemma}\label{SRG}
Let $G$ be a walk-regular and co-edge-regular graph. If there exists a vertex $x\in V(G)$ such that $a_{xy}=a_{xy'}$ for all $y,y'\in N_G(x)$, then $G$ is strongly regular.
\end{lemma}
\begin{proof}
Let $G$ be a walk-regular and co-edge-regular graph. Let $x$ be a vertex in $G$, such that $a_{xy}=a$ for all $y\in N_G(x)$, where $a$ is a constant. Let $u$ be a vertex in $G$. We now show that $a_{uv}=a$ for all $v\in N_G(u)$. Note that, $\sum\limits_{v\sim u}a_{uv}=\sum\limits_{y\sim x}a_{xy}$ and $\sum\limits_{v\sim u}a_{uv}^2=\sum\limits_{y\sim x}a_{xy}^2$, by Lemma \ref{twoconstants}. Hence,
\begin{equation*}
\begin{array}{rl}
  \sum\limits_{v\sim u}(a_{uv}-a)^2 & = \sum\limits_{v\sim u}a_{uv}^2-2a\sum\limits_{v\sim u}a_{uv}+a^2 \\
   & =\sum\limits_{y\sim x}a_{xy}^2-2a\sum\limits_{y\sim x}a_{xy}+a^2 \\
   & =\sum\limits_{y\sim x}(a_{xy}-a)^2=0.
\end{array}
\end{equation*}
This shows the lemma.
\end{proof}

\begin{lemma}\label{clique}
Let $G$ be a strongly co-edge-regular graph with parameters $(n,k,2,\ell)$. Let $x$ be a vertex in $G$ and let $W:=\{w\mid x\sim w, a_{xw}\geqslant \frac{k}{2}\}$. If $W\neq\emptyset$, then $W$ induces a clique in $G$.
\end{lemma}
\begin{proof}
Let $G$ be a strongly co-edge-regular graph with parameters $(n,k,c,\ell)$, where $c=2$. Let $x$ be a vertex in $G$. Let $W:=\{w\mid x\sim w, a_{xw}\geqslant \frac{k}{2}\}$. The lemma is clear when $|W|=1$. Now we assume $|W|\geqslant2$. Take $w_1,w_2$ in $W$. Suppose that $w_1,w_2$ are not adjacent. Note that $\{w_1,w_2\}\cup N_G(x,w_1) \cup N_G(x,w_2)\subseteq N_G(x)$. Then,
\begin{equation*}
  |N_G(x,w_1)\cap N_G(x,w_2)|\geqslant 2+a_{xw_1}+a_{xw_2}-k\geqslant 2.
\end{equation*}
 This means $w_1$ and $w_2$ have at least $3$ common neighbours in $G$, as $\{x\}\cup (N_G(x,w_1)\cap N_G(x,w_2))\subseteq N_G(w_1,w_2)$. This is a contradiction, as $c=2$. This shows the lemma.
\end{proof}

The following theorem shows that a strongly co-edge-regular graph with large clique number has large $\ell$.

\begin{theorem}\label{largec1}
  Let $G$ be a strongly co-edge-regular graph with parameters $(n,k,2,\ell)$ and clique number $\omega$. If $\omega>\frac{\ell+4}{2}$, then $\ell=k-2$.
\end{theorem}
\begin{proof}
Let $G$ be a strongly co-edge-regular graph with parameters $(n,k,c,\ell)$ and clique number $\omega$, where $c=2$ and $\omega>\frac{\ell+4}{2}$. Let $x$ be a vertex in a maximum clique in $G$. Denote by $\Delta(x)$ the local graph of $x$ in $G$. Assume $C$ is a maximum clique in $\Delta(x)$ with order $\omega'=\omega-1>\frac{\ell+2}{2}$. Define $R:=N_G(x)-C$ and $r:=|R|$.
\begin{claim}\label{noedgeCandR}
There is no edge between $C$ and $R$.
\end{claim}
\noindent{\bf Proof of Claim \ref{noedgeCandR}.} Assume there exists $u\in C$ and $v\in R$ such that $u\sim v$. There exists a vertex $u'\in C$ with $u\neq u'$ such that $u'\not\sim v$, as $C$ is a maximum clique.
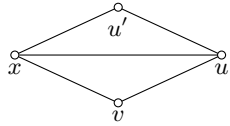
\begin{figure}[h]
   \centering
    \begin{tikzpicture}
    \tikzstyle{every node}=[draw,circle,fill=white,minimum size=3pt,
                            inner sep=0pt]
                            {every label}=[\footnotesize]
   \draw (-1.299,0) node (1) [label=below:$x$] {}
       -- ++(25:1.5cm) node (2) [label=below:$u'$] {}
       -- ++(-25:1.5cm) node (3) [label=below:$u$] {}
       -- ++(205:1.5cm) node (4) [label=below:$v$] {}
       -- ++ (1);
   \draw (3) -- (1);
    \end{tikzpicture}
    \end{figure}

Note that
\begin{equation*}
\ell=a_{u'x}+a_{u'u}\geqslant2(\omega'-1)>2(\frac{\ell+2}{2}-1)=\ell,
\end{equation*}
as $|C|=\omega'>\frac{\ell+2}{2}$. This is a contradiction, which shows Claim \ref{noedgeCandR}.
\qed

Let $u$ be a vertex in $C$. Define
\begin{equation*}
W(u):=\{w \mid w\sim u, w\not\sim x\}.
\end{equation*}
By Claim \ref{noedgeCandR}, we obtain $|W(u)|=k-1-(\omega'-1)=k-\omega'=r$. Note that every vertex in $W(u)$ has exactly one neighbour in $C$, as $\ell-a_{xu}=\ell-(\omega'-1)<\ell-\frac{\ell}{2}<\frac{\ell}{2}<\omega'-1$. Then every vertex in $W(u)$ has exactly one neighbour in $R$, as $c=2$. By Claim \ref{noedgeCandR}, the vertex $v$ has no neighbours in $C$ for $v\in R$. Hence, $v$ has one neighbour in $W(u)$. So, $a_{xv}=\ell-a_{xu}<\frac{\ell}{2}$ for $v\in R$. It follows that any two distinct vertices in $R$ have no common neighbours outside $N_G(x)\cup\{x\}$.

\begin{claim}\label{Risclique}
The set $R$ forms a clique in $G$.
\end{claim}
\noindent{\bf Proof of Claim \ref{Risclique}.} Suppose that $v_1\neq v_2\in R$ are not adjacent. As $a_{xv_1}=a_{xv_2}=\ell-a_{xu}<\frac{\ell}{2}$ and $c=2$, we obtain $v_1$ and $v_2$ have a common neighbour $v_3$ in $N_G(x)$. By Claim \ref{noedgeCandR}, the vertex $v_3$ is in $R$.
\begin{figure}[h]
   \centering
    \begin{tikzpicture}
    \tikzstyle{every node}=[draw,circle,fill=white,minimum size=3pt,
                            inner sep=0pt]
                            {every label}=[\footnotesize]
   \draw (-1.299,0) node (1) [label=below:$x$] {}
       -- ++(25:1.5cm) node (2) [label=below:$v_1$] {}
       -- ++(-25:1.5cm) node (3) [label=below:$v_3$] {}
       -- ++(205:1.5cm) node (4) [label=below:$v_2$] {}
       -- ++ (1);
   \draw (3) -- (1);
    \end{tikzpicture}
    \end{figure}
As $a_{xv_1}=a_{xv_3}=\ell-a_{xu}<\frac{\ell}{2}$, the vertices $v_1$ and $v_3$ have no common neighbours outside $N_G(x)-\{x\}$. Hence, $a_{v_1v_3}\leqslant r-2+1<\frac{\ell}{2}$. It follows that
\begin{equation*}
\ell=a_{v_1x}+a_{v_1v_3}<\frac{\ell}{2}+\frac{\ell}{2}=\ell,
\end{equation*}
which is a contradiction. This shows Claim \ref{Risclique}.
\qed

Let $u$ be a vertex in $C$. Note that $a_{xu}=\omega'-1>\frac{\ell}{2}$, by Claim \ref{noedgeCandR}. Then there exists a vertex $v$ in $R$, such that $a_{xv}+a_{xu}=\ell$. By Claims \ref{noedgeCandR} and \ref{Risclique}, we have $a_{xv}=k-\omega'-1$. Hence, $\ell=k-2$. This finishes the proof of Theorem \ref{largec1}.
\end{proof}

\section{Main results}\label{MainResults}
\subsection{Strongly co-edge-regular graphs with large $\ell$}\label{4.1}

In this subsection, we show that a walk-regular and strongly co-edge-regular graph with parameters $(n,k,2,\ell)$ is a $p\times q$-grid, where $p+q=k+2$ or the $2$-clique extension of pentagon, if $\ell\geqslant \frac{3}{4}k$. Moreover, we show that there does not exist a strongly co-edge-regular graph with parameters $(n,k,2,\ell)$ and smallest eigenvalue at least $-3$, satisfying $k\geqslant120$ and $\ell<\frac{3}{4}k$.

First we consider the case where $G$ is a walk-regular and strongly co-edge-regular graph with $\ell=k-2$.

\begin{theorem}\label{l=k-2}
Let $G$ be a walk-regular and strongly co-edge-regular graph with parameters $(n,k$, $2,k-2)$. If $G$ contains an induced quadrangle, then $G$ is isomorphic to the Shrikhande graph or a $p\times q$-grid, where $p+q=k+2$.
\end{theorem}
\begin{proof}
Let $G$ be a walk-regular and strongly co-edge-regular graph with parameters $(n,k,c,\ell)$, such that $\ell=k-2$ and $c=2$. Assume that $G$ contains an induced quadrangle, say $x\sim u\sim y \sim v \sim x$. Assume further, the number of common neighbours of $x,u$ and $x,v$ satisfy $s=a_{xu}\geqslant a_{xv}=t$. As $c=2$, we have $N_G(x,u)\cap N_G(x,v)=\emptyset$ and $s+t=a_{xu}+a_{xv}=\ell=k-2$. Note that $a_{yu}=\ell-a_{xu}=\ell-s=t$ and $a_{yv}=\ell-a_{xv}=\ell-t=s$.
\begin{claim}\label{sort}
The number $a_{xw}\in\{s,t\}$ for any $w\in N_G(x)$.
\end{claim}
\noindent{\bf Proof of Claim \ref{sort}.} Let $w_1$ be a vertex in $N_G(x)-\{u,v\}$. By $a_{xu}+a_{xv}=\ell=k-2$, we have $w_1$ is adjacent to $u$ or $v$. Without loss of generality, we may assume $w_1\sim u$. As $c=2$, we have $w_1\not\sim y$. Then there exists a vertex $w_2\in N_G(w_1,y)-\{u\}$. By $a_{yu}+a_{yv}=\ell=k-2$, we have $w_2$ is adjacent to $u$ or $v$. Note that $w_2$ is not adjacent to $x$, as $c=2$. So we find
\begin{equation*}
a_{xw_1}=\left\{
 \begin{array}{rl}
\ell-a_{xu}=t,& \text{if $w_2\sim u$},\\
\ell-a_{xv}=s,& \text{if $w_2\sim v$}.
 \end{array}
 \right.
\end{equation*}
This shows the claim.
\qed

Now we consider the following two cases.

{\bf Case 1.} $s=t$.

In this case, by Lemma \ref{SRG}, we have $G$ is a strongly regular graph with parameters $(n,k,\frac{k-2}{2},2)$. Then by Lemma \ref{SRGspectrum}, the three distinct eigenvalues of $G$ are $k, \frac{k-2}{2}$ and $-2$. Hence, $G$ is isomorphic to the Shrikhande graph or the $(s+2)\times (s+2)$-grid, by Lemma \ref{atleast-2}.

{\bf Case 2.} $s>t$.

Let $w$ be a vertex in $N_G(x,v)$. By $c=2$, we obtain $w\not\sim u$ and $w$ has at most one neighbour in $N_G(x,u)$ for all $w\in N_G(x,v)$, and similarly, for all $w'\in N_G(x,u)$, we have $w'\not\sim v$ and $w'$ has at most one neighbour in $N_G(x,v)$. Then there exists at least one vertex $z\in N_G(x,u)$ such that $z$ has no neighbours in $N_G(x,v)$, as $a_{xu}=s>t=a_{xv}$. Note that $z\not\sim v$. There exists a vertex $z_1$ such that $N_G(z,v)=\{x,z_1\}$. Note, $z_1$ is not adjacent to $x$, as $z$ has no neighbour in $N_G(x,v)$. Hence, $a_{xz}=\ell-a_{xv}=s$, which implies that $z$ is adjacent to all vertices in $N_G(x,u)-\{z\}$. Therefore, any two distinct vertices in $N_G(x,u)-\{z\}$ have at least three common neighbours, namely, $x,u,z$. By $c=2$, we obtain that the subgraph induced on $\{u\}\cup N_G(x,u)$ is a clique with valency $s$. It follows that there are no edges between $N_G(x,u)$ and $N_G(x,v)$. We find $a_{xw}\leqslant 1+|N_G(x,v)-\{w\}|\leqslant a_{xv}=t$ for all $w\in N_G(x,v)$. By Claim \ref{sort}, we have $a_{xw}=t$ for $w\in N_G(x,v)$. This shows that the subgraph induced on $\{v\}\cup N_G(x,v)$ is a clique with valency $t$. Hence, the local graph of $x$ in $G$ is isomorphic to $K_{s+1}\dot\cup K_{t+1}$.

Let $u'$ (resp. $v'$) be a vertex in $N_G(x,u)$ (resp. $N_G(x,v)$), and $x,y'$ be the common neighbours of $u',v'$. Then the subgraph induced on $\{x,u',y',v'\}$ is a quadrangle. Hence, every neighbour of $x$ lies on a quadrangle. It follows that every vertex in $G$ lies on a quadrangle. This shows that the local graph of any vertex of $G$ is isomorphic to $K_{s+1}\dot\cup K_{t+1}$. This shows that $G$ is a $(s+2)\times (t+2)$-grid.

This completes the proof of the theorem.
\end{proof}

\begin{theorem}\label{atleast3/4k}
There does not exist a strongly co-edge-regular graph with parameters $(n,k,2,\ell)$ which contains an induced quadrangle and satisfies $\frac{3k}{4}\leqslant \ell\leqslant k-3$.
\end{theorem}
\begin{proof}
Let $G$ be a strongly co-edge-regular graph with parameters $(n,k,c,\ell)$, such that $\frac{3k}{4}\leqslant \ell= k-d$ and $c=2$, where $\frac{k}{4}\geqslant d\geqslant3$ is an integer. Assume that $G$ contains an induced quadrangle, say $x\sim u\sim y \sim v \sim x$. Without loss of generality, we may assume that $a_{xu}\geqslant a_{xv}$. Define
\begin{equation*}
  \begin{array}{rl}
    U :=& N_G(x,u)\cup\{u\}, \\
    V :=& N_G(x,v)\cup\{v\}, \\
    W :=& N_G(x)-(U\cup V), \\
    Z :=& \cup_{w\in W}N_G(w,y).
  \end{array}
\end{equation*}

Note that $|W|=k-2-\ell=d-2$, as $\ell=k-d$. As $c=2$, the set $W\cap N_G(y)$ is empty and every vertex in $W$ has two neighbours in $Z$.
\begin{claim}\label{axw=axv}
If $\ell>2(d-1)$, then for every $w\in W$, $a_{xw}=a_{xv}$ holds, and for all $u'\in U$ we have $w\not\sim u'$.
\end{claim}
\noindent{\bf Proof of Claim \ref{axw=axv}.} Assume that $\ell>2(d-1)$. Note that, for $w\in W$, we have $a_{xw}\leqslant k-(\ell+2)-1+2=d-1$, as $x,w,u$ (resp. $x,w,v$) have at most one common neighbour. Then for $z\in Z$, the vertex $z$ has exactly one neighbour in $W$ as $\ell>2(d-1)$. Hence, $|Z|=2|W|=2(d-2)$. Since $z\not\sim x$ and $a_{xv}+a_{xw}\leqslant \frac{\ell}{2}+d-1<\ell$, we obtain $z\not\sim v$ for $z\in Z$.
As \begin{equation*}
\begin{array}{rcl}
  k&= & |W|+2+a_{xu}+a_{xv}=|W|+2+a_{yu}+a_{yv}, \\
  |Z|&= & 2|W|=2(d-2)
\end{array}
\end{equation*}
and for all $z\in Z$ $z\not\sim v$, we find $|N_G(y,u)\cap Z|\geqslant d-2$.

For $w\in W$, the vertices $u$ and $w$ have at most one common neighbour in $Z$, as $c=2$ and $x$ is a common neighbour of $u$ and $w$. Thus, $|N_G(y,u)\cap Z|=|W|=d-2$ and $w,u$ have exactly one common neighbour in $Z$ for $w\in W$. It follows that $a_{xw}=\ell-a_{xu}=a_{xv}$ for $w\in W$.
Moreover, the vertices $w$ and $u'$ are non-adjacent for $u'\in U$ and $w\in W$.
This finishes the proof of Claim \ref{axw=axv}.
\qed

As $\ell \geqslant \frac{3k}{4}$ and $d \leqslant \frac{k}{4}$, we have $\ell \geqslant 3d > 2(d-1)$.
Fix $w \in W$. Then $w$ and $v$ have at most one neighbour in $V$, as $c=2$.
This means that $a_{xv} = a_{xw} \leqslant |W| = d-2$, by Claim \ref{axw=axv}.
It follows that $a_{xu} = \ell - a_{xv} \geqslant \ell - (d-2) > \frac{3k}{4}  -\frac{k}{4} = \frac{k}{2}$.
By Lemma \ref{clique}, the set $U' = \{u' \sim x \mid a_{u'x} = a_{ux}\}$ is a clique. In particular, $U' \subseteq U$ holds.
Let $Z' := \{ z' \neq x \mid z' \not\sim x, z' \sim w\}$. We find
$|Z'| \geqslant  k - a_{xw} \geqslant k -d +2.$
As $z'\in Z'$ has exactly one neighbour in $U'$  and $u' \in U'$ has at most one neighbour in $Z'$, we find that
$|U| \geqslant |U'| \geqslant |Z'| \geqslant k-d+2.$
On the other hand
\begin{equation*}
|U| \leqslant k - a_{xv} -1 - |W| \leqslant k - 1- d+2 = k-d+1,
\end{equation*}
which is a contradiction. This finishes the proof of the theorem.
\end{proof}

Note that Theorem \ref{grid} immediately follows from Lemma \ref{noquadrangles}, Theorem \ref{l=k-2} and Theorem \ref{atleast3/4k}. 
Now we show that, if $G$ is a strongly co-edge-regular graph with parameters $(n,k,2,\ell)$ and smallest eigenvalue at least $-3$, then $G$ is a $t\times s$-grid or $k$ is small.

\begin{theorem}\label{l<3k/4}
There is no strongly co-edge-regular graph with smallest eigenvalue $\theta_{\min}$ and parameters $(n,k,2,\ell)$ satisfying $\theta_{\min}\geqslant-3$, $k\geqslant 120$ and $\ell<\frac{3k}{4}$.
\end{theorem}
\begin{proof}
Let $G$ be a strongly co-edge-regular graph with parameters $(n,k,2,\ell)$, satisfying $k\geqslant120$ and $\theta_{\min}\geqslant-3$. Let $x$ be a vertex in $G$. Consider $\Delta(x)$, the local graph of $x$ in $G$. Let $\{z_1,\ldots,z_t\}$ be a maximum independent set in $\Delta(x)$. Define
\begin{equation*}
\begin{array}{rl}
C_i:=&\{w\mid w\sim x,\ w\sim z_i,\ w\not\sim z_j,\ j\neq i\}\cup\{z_i\},\ i=1,\ldots,t,\\
R:=&N_G(x)-\cup_{i=1}^{t}C_i.
\end{array}
\end{equation*}
Then $C_i$ forms a clique for $i=1,\ldots,t$, as $\{z_1,\ldots,z_t\}$ is a maximum independent set. Let $c_i:=|C_i|$ for $i=1,\ldots, t$ and $r=|R|$. Without loss of generality, we may assume $c_1\geqslant\cdots\geqslant c_t$. Note that every vertex in $R$ has at least two neighbours in $\{z_1,\ldots,z_t\}$. Then $r\leqslant\binom{t}{2}$, as $c=2$.

\begin{claim}\label{oneedge}
There exists at most one edge between $C_i$ and $C_j$ for $1\leqslant i\neq j\leqslant t$.
\end{claim}
\noindent{\bf Proof of Claim \ref{oneedge}.} Note that every vertex in $C_i$ has at most one neighbour in $C_j$ and vice versa, as $c=2$. Suppose there are two disjoint edges between $C_i$ and $C_j$. Then $\Delta(x)$ contains an induced quadrangle. This is impossible, as $c=2$. This shows Claim \ref{oneedge}.
\qed

\begin{claim}\label{t=3}
We have $t=3$.
\end{claim}
\noindent{\bf Proof of Claim \ref{t=3}.} First, we show that $t\leqslant8$ holds. Let $v$ be a vertex in $G$, such that $v\sim z_1$ and $v\not\sim x$. Then the subgraph of $G$ induced on $\{x,v,z_1,\ldots,z_t\}$ is bipartite. As $\theta_{\min}\geqslant-3$, we obtain $t\leqslant8$, by Lemma \ref{forbiddensubgraphs} ($\mathrm{i}$).
By Claim \ref{oneedge}, there exist at least $c_2-1$ vertices in $C_2$ which have no neighbours in $C_1$.
Note that by Theorem \ref{largec1}, we have $c_1\leqslant\frac{\ell+2}{2}<\frac{3k+8}{8}$, as $\ell<\frac{3k}{4}<k-2$.

Hence, we obtain
\begin{equation}\label{c_2}
\begin{array}{rl}
  c_2& \geqslant\frac{k-c_1-r}{t-1} \\
   & >\frac{5k-8-4t(t-1)}{8(t-1)} \\
   & \geqslant\frac{5\times 120-8-4t(t-1)}{8(t-1)} \\
   &\geqslant\left\{
            \begin{array}{rl}
              6.5, & \text{if } t=8, \\
              8.8, &  \text{if } t=7, \\
              11.8, & \text{if } t=6, \\
              16, &  \text{if } t=5, \\
              22.6, &  \text{if } t=4,
            \end{array}
          \right.
\end{array}
\end{equation}
as $k\geqslant120$.

Now we show that $t\leqslant5$. Assume that $t\geqslant6$. There exists at most one edge between $C_1$ and $C_2$, by Claim \ref{oneedge}. It follows that $C(2K_6\dot\cup 6K_1)$, $C(2K_8\dot\cup 5K_1)$ or $C(2K_{11}\dot\cup 4K_1)$ is a subgraph of $G$ induced on a subset of $C_1\cup C_2\cup\{x,z_3,\ldots,z_t\}$ for $t=8,7$ and $6$, respectively. This is a contradiction, by Lemma \ref{forbiddensubgraphs} ($\mathrm{ii}$).


For $4\leqslant t\leqslant 5$, we have
\begin{equation}\label{c_3}
\begin{array}{rl}
  c_3 & \geqslant\frac{k-c_1-c_2-r}{t-2} \\
   & >\frac{k-8-2t(t-1)}{4(t-2)} \\
   & \geqslant\frac{120-8-2t(t-1)}{4(t-2)} \\
   &=\left\{
            \begin{array}{rl}
              6, &  \text{if } t=5, \\
              11, &  \text{if } t=4,
            \end{array}
          \right.
   \end{array}
\end{equation}
as $c_2\leqslant c_1<\frac{3k+8}{8}$ and $k\geqslant120$.

There exists at most one edge between $C_i$ and $C_j$ for $1\leqslant i\neq j\leqslant3$, by Claim \ref{oneedge}. It follows from (\ref{c_2}) and (\ref{c_3}) that 
\begin{equation*}
         \begin{array}{ll}
         c_1\geqslant c_2 \geqslant 17\text{ and }c_3\geqslant 7,&\text{ if }t=5,\text{ and} \\
         c_1\geqslant c_2 \geqslant 23\text{ and }c_3\geqslant 12,&\text{ if }t=4.
       \end{array}
       \end{equation*}

This shows that $C(2K_{15}\dot\cup K_7\dot\cup 2K_1)$ or $C(2K_{21}\dot\cup K_{12}\dot\cup K_1)$ is an induced subgraph of $G$, and this is a contradiction, by Lemma \ref{forbiddensubgraphs} ($\mathrm{iii}$). This finishes the proof of Claim \ref{t=3}.
\qed

\begin{claim}\label{localgraph}
The local graph of $x$ in $G$, $\Delta(x)$, is isomorphic to $3K_{\frac{k}{3}}$.
\end{claim}
\noindent{\bf Proof of Claim \ref{localgraph}.} We have $N_G(x)=C_1\cup C_2\cup C_3\cup R$, by Claim \ref{t=3}. Now we consider two cases, namely, $R=\emptyset$ and $R\neq\emptyset$.

First, we assume that $R=\emptyset$. As $c=2$, we obtain $\ell=a_{xz_i}+a_{xz_j}=(c_i-1)+(c_j-1)=c_i+c_j-2$ for $1\leqslant i<j \leqslant3$. Hence, $c_1=c_2=c_3=\frac{k}{3}$ and $\ell=2\times\frac{k-3}{3}=\frac{2k-6}{3}$. Let $v$ be a vertex in $\Delta(x)$. Then $a_{xv}\geqslant\frac{k-3}{3}$. By Claim \ref{oneedge}, we have $a_{xv}\leqslant (c_1-1)+2=\frac{k+3}{3}<k-1$, as $k\geqslant120$. This means $v$ has a neighbour $z$ outside $\{x\}\cup N_G(x)$.

Let $u,v$ be the two common neighbours of vertices $x$ and $z$. Then
\begin{equation*}
\frac{2k-6}{3}=\ell=a_{xu}+a_{xv}\geqslant \frac{k-3}{3}+\frac{k-3}{3}=\frac{2k-6}{3}.
\end{equation*}
This means that $a_{xv}=\frac{k-3}{3}$, and there exist no edges between $C_i$ and $C_j$ for $1\leqslant i,j\leqslant3$. Therefore, $\Delta(x)\cong 3K_{\frac{k}{3}}$.

Now assume that $R\neq\emptyset$. Let $w$ be a vertex in $R$. As $c=2$, note that $w$ either has at most one neighbour in $C_i$ or is adjacent to all vertices in $C_i$ for each $i=1,2,3$.  By Theorem \ref{largec1}, we have $c_1+c_2\leqslant\frac{\ell+2}{2}+\frac{\ell+2}{2}=\ell+2$. Then
\begin{equation*}
  c_3=k-(c_1+c_2+r)\geqslant k-(\ell+2+\binom{t}{2})>k-\frac{3k}{4}-5=\frac{k}{4}-5\geqslant 25,
\end{equation*}
as $t=3$, $\ell<\frac{3k}{4}$ and $k\geqslant 120$. As there is at most one edge between $C_i$ and $C_j$, we see that there exists clique $C_i'\subseteq C_i$ of order $c_i-1$ such that the induced subgraph on $V(C_1')\cup V(C_2')\cup V(C_3')$ is $K_{c_1-1}\dot\cup K_{c_2-1}\dot\cup K_{c_3-1}$. Note $c_1-1\geqslant c_2-1\geqslant c_3-1\geqslant 24$. If $w$ has at most one neighbour in each $C_i$ for $i=1,2,3$, then we can find $3$ vertices $w_1'\in C_1'$, $w_2'\in C_2'$ and $w_3'\in C_3'$, such that $w_i'\not\sim w_j'$ for $i\neq j$ and $w_i'\not\sim w$ for all $i=1,2,3$. So we are back to the case $t\geqslant 4$. If $w$ is adjacent to all vertices of $C_1\dot\cup C_2\dot\cup C_3$, then we see that $\Delta(x)$ contains an induced $C(3K_5)$. This is a contradiction with Lemma \ref{forbiddensubgraphs} ($\mathrm{iii}$). In similar fashion, we obtain that there exists an induced $K_{13}\cup C(2K_{13})$ in $\Delta(x)$, if $w$ is adjacent to all vertices of $C_i\dot\cup C_j$ for some $i,j$ satisfying $1\leqslant i<j\leqslant 3$. Hence, $w$ is adjacent to all vertices of $C_i$ for exactly one $i\in\{1,2,3\}$. Take $z_i'\in C_i-\{z_i\}$ such that $z_i'$ has no neighbours in $C_j$ for $i,j\in\{1,2,3\}$ such that $i\neq j$. Then each vertex in $\Delta(x)-\{z_1',z_2',z_3'\}$ has exactly one of $\{z_1',z_2',z_3'\}$ as its neighbour. Hence, we are back to the case $R=\emptyset$.
This finishes the proof of Claim \ref{localgraph}.
\qed

Since $\Delta(x)\cong 3K_{\frac{k}{3}}$ for $x\in V(G)$, we obtain that $G$ is strongly regular with parameters $(\frac{k^2+3k+3}{3}$, $k,\frac{k-3}{3},2)$. Assume $G$ has eigenvalues $k>\theta>\tau$ with respective multiplicities $1,m_\theta,m_\tau$. By Lemma \ref{SRGspectrum}, we have
\begin{equation*}
  m_\tau-m_\theta=\frac{2k+(\frac{k^2+3k+3}{3}-1)(\frac{k-3}{3}-2)}{\sqrt{(\frac{k-3}{3}-2)^2+4(k-2)}}.
\end{equation*}
Note that $ m_\tau-m_\theta$ is an integer, as $k\geqslant120$. It follows that $(\frac{k-3}{3}-2)^2+4(k-2)=\frac{k^2+18k+9}{9}$ must be a perfect square. This implies $k=0$, which is a contradiction.

This finishes the proof of Theorem \ref{l<3k/4}.
\end{proof}

Now, we are ready to prove Theorem \ref{main1}.

\noindent{\bf Proof of Theorem \ref{main1}.}
Let $G$ be a walk-regular and strongly co-edge-regular graph with parameters $(n,k,2,\ell)$, such that $k\geqslant120$. Assume that $\theta_{\min}\geqslant-3$. By Theorem \ref{l<3k/4} we see that no graph exists for $l < \frac{3}{4}k$. For $l \geq \frac{3}{4}k$, G is a $p \times q$-grid, where $p+q=k+2$ and $l = k-2$ by Theorem \ref{grid}, as $k \geq 120$. This completes the proof of this theorem.

\subsection{Co-edge-regular graphs with four eigenvalues}\label{4.2}
In this subsection we study co-edge-regular graphs with four eigenvalues.
We start with the following lemma.
\begin{lemma}\label{four}
Let $G$ be a connected regular graph with $n$ vertices and valency $k$. If $G$
has exactly four distinct eigenvalues $\{\theta_0=k, \theta_1, \theta_2, \theta_3\}$, then $G$ is walk-regular.
If moreover $G$ is co-edge-regular with parameters $(n, k, c)$, then $G$ is strongly co-edge-regular with parameters $(n, k, c, \ell)$ where $\ell=(\sum\limits_{i=1}^3 \theta_i )c +\frac{\prod_{i=1}^3(k-\theta_i)}{n} - (k-c)c.$
\end{lemma}
\begin{proof}
Let $G$ be a connected regular graph with $n$ vertices and valency $k$ having exactly four distinct eigenvalues $\{\theta_0=k, \theta_1, \theta_2, \theta_3\}$.
Then the adjacency matrix $A$ of $G$ satisfies the following equation (see \cite{hoff}):
\begin{align}\label{eq2.3}
A^3-( \sum\limits_{i=1}^3 \theta_i )A^2+ ( \sum_{1\leq i<j\leq 3}\theta_i\theta_j ) A-\theta_1\theta_2\theta_3I=\frac{\prod_{i=1}^3(k-\theta_i)}{n}J.
\end{align}
This implies that $G$ is walk-regular, as was shown by Van Dam \cite{vanDam1995}.
Now assume that $G$ is also co-edge-regular with parameters $(n, k, c)$.
Let $x, y$ be two vertices at distance 2.
Then Equation (\ref{eq2.3}) gives us
$$(A^3)_{xy} = (\sum\limits_{i=1}^3 \theta_i )c +\frac{\prod_{i=1}^3(k-\theta_i)}{n}.$$
This implies
$$\sum_z a_{xz} + (k-c)c = (\sum\limits_{i=1}^3 \theta_i )c +\frac{\prod_{i=1}^3(k-\theta_i)}{n},$$ where the first sum is taken over all common neighbours $z$ of $x$ and $y$.
It follows that $G$ is strongly co-edge-regular with parameters $(n, k, c, \ell)$ where $\ell=(\sum\limits_{i=1}^3 \theta_i )c +\frac{\prod_{i=1}^3(k-\theta_i)}{n} - (k-c)c.$
This shows the lemma.
\end{proof}

As an immediate  consequence of Theorems \ref{grid} and \ref{main1} and Lemma \ref{four}, we obtain Theorem \ref{fourev}.

\section*{Acknowledgments}
\indent

We would like to thank the referees for their detailed comments. Their comments significantly improved the paper.

Brhane Gebremichel is supported by a Chinese Scholarship Council at University of Science and Technology of China.

We greatly thank professor Min Xu supporting M.-Y. Cao to visit University of Science and Technology of China.

J.H. Koolen is partially supported by the National Natural Science Foundation of China (No. 12071454), Anhui Initiative in Quantum Information Technologies (No. AHY150000), and
the project "Analysis and Geometry on Bundles" of Ministry of Science and Technology of the People's Republic of China.

\bibliographystyle{plain}
\bibliography{BCK}

\end{document}